\documentclass[11pt,reqno]{amsart}
\usepackage[top=3cm,bottom=3cm,left=3cm,right=3cm]{geometry}
\usepackage{amsthm,amsmath,amsfonts,amssymb}
\usepackage{hyperref} 

\newtheorem{Theorem}{Theorem}
\newtheorem{Corollary}[Theorem]{Corollary}
\newtheorem{Lemma}[Theorem]{Lemma}
\newtheorem{Proposition}[Theorem]{Proposition}
\newtheorem{Remark}[Theorem]{Remark}
\newtheorem{Definition}[Theorem]{Definition}

\begin{document}
\title[Robustness for nonlocal eq. under perturbation]{On the robustness of pullback attractors for a nonlocal reaction-diffusion equation under perturbation}

\author[Rub\'en Caballero, Pedro Mar\'in-Rubio and Jos\'e Valero]{Rub\'en Caballero, Pedro Mar\'in-Rubio and Jos\'e Valero}

\address[Rub\'en Caballero and Jos\'e Valero]{Centro de Investigaci\'on Operativa, Universidad Miguel Hern\'andez de Elche,\\
Avda. Universidad s/n, 03202, Elche (Alicante), Spain}
\email[R. Caballero]{ruben.caballero@umh.es}
\email[J. Valero]{jvalero@umh.es}

\address[Pedro Mar\'in-Rubio]{Departamento de Ecuaciones Diferenciales y An\'alisis Num\'erico\\
Universidad de Sevilla\\
C/ Tarfia s/n, 41012, Sevilla, Spain}
\email[P. Mar\'in-Rubio]{pmr@us.es}

\date{}



\maketitle

\begin{abstract} A parametric family of reaction-diffusion equations with nonlocal viscosity is considered. Existence of solutions and actually of pullback attractors is known from previous works. In this paper we obtain a robustness result of the attractors toward the corresponding minimal pullback attractor of the limiting problem. This result extends the ones obtained in \cite{ND16}. Actually here all terms (reactions, external forces and nonlocal viscosity functions) may vary with the parameter. The upper semicontinuous convergence of attractors is obtained under rather general assumptions and in a fully non-autonomous context using the framework of tempered universes. 
\end{abstract}

\begin{center}Dedicated to Professor Peter Kloeden on occasion of his Seventieth Birthday\end{center}

\section{Introduction} 

The theory of parabolic PDEs (heat transfer and general concentration diffusion among others) has experimented a great increase of nonlocal models in the last few decades (all throughout this paper by nonlocal we only mean in space). Actually there are many examples in several sciences that require different types of nonlocal modeling, e.g. in Physics one may refer to phase-field analysis (introducing the phase variable) for phase-transition of material (single or combined), or in Biology where the number of equations and systems for tumor growth analysis has increased exponentially in the last years, some of them taking into account nonlocal effects, convolution operators as fractional laplacian and others, or bacteria group movements as macro organisms (e.g. cf. \cite{25,40,brr}). Namely these models include terms that, instead of a local measure in a certain point, are in fact nonlocal as an average (isotropic or anisotropic) of the unknown in a neighborhood of each point. This is sometimes due to the influence that the neighborhood has on each point, and some measurements come from taking experimental data.

The following nonlocal diffusion equation (fulfilled with suitable initial and boundary conditions) has been considered by many authors:
$$\frac{\partial u}{\partial t}-a(l(u))\Delta u=f\quad\textrm{in $\Omega$,}$$
where $a$ is a certain positive function, $l\in (L^2(\Omega))'$ is a functional acting on $u$ over the whole domain $\Omega,$ i.e. $l(u(t)):=\int_\Omega g(x)u(x,t)dx$ for a certain $g\in L^2(\Omega)$. It is clear that $a$ is a viscosity which takes into account the weighted average of $u.$ Depending on the increasing or non-increasing character of $a$ one may simulate aggregation effects or the opposite, leaving crowded zones behavior. Several papers by Chipot and his collaborators (e.g. cf. \cite{chipot-lovat,chipot-molinet}) use the former in epidemic theory or to study heat propagation (also in divergence form for inhomogeneous domains \cite{chipot-siegwart}; actually we can develop the same analysis with an operator in divergence form as in the cited reference, but for simplicity in the presentation we keep the laplacian). The mathematical analysis is not restricted to existence and uniqueness of solutions, but also refers to stationary points, convergence of evolutive to stationary solutions, ordered intervals, or general stability issues among others. Some variations are also known to have Lyapunov functionals (e.g. cf. \cite{chipot-valente-caffarelli,chipot-zheng,chipot-savitska}), but this is not the general case, implying a more complex analysis. 

The case of $f$ depending on the unknown $u$ has also been recently treated in some forms. For an interesting nonlocal reaction part with small values of a parameter we may refer to \cite{4}. The case of nonlocal viscosity and $f(u)$ 
$$\frac{\partial u}{\partial t}-a(l(u))\Delta u=f(u)+h\quad\textrm{in $\Omega$}$$
has been developed in some recent papers (see \cite{NA15} for a simpler approach with sublinear term and \cite{ND16,DCDSB18,PRSE18,JMAA18} for proper general nonlocal reaction-diffusion models in several situations). Existence, sometimes uniqueness, regularity and attractors issues have been addressed in the above references. No need to say that the analysis of stationary points and decay is extremely difficult since again the obtention of a Lyapunov functional is not obvious at all neither the study of stationary points. This drawback makes the study of existence and properties of attractors even more interesting as a natural extension. 

On other hand, a main concern in the study of a model is its continuous behavior with respect to some of its elements. Suppose for instance that real data are not available in a straightforward way but collected successively and that instead of a single problem, we have a family of problems with analogous structure but slightly different (readjusted) terms that we notate with a parameterized index. In this sense, both continuity in finite-time intervals and robustness properties -when hold- indicate how some structures vary (at least) continuously w.r.t. parameters (e.g. cf. \cite{marin-rubio} for similar results in a setting with delay). Namely we consider the following perturbed family of reaction-diffusion equations
$$(P_\eta)\left\{\begin{array}{ll}
\displaystyle\frac{\partial u}{\partial t}-a_\eta(l_\eta(u))\Delta u=f_\eta(u)+h_\eta(t)&\textrm{in $\Omega\times (\tau,\infty)$},\\
u=0&\textrm{on $\partial\Omega\times (\tau,\infty)$},\\
u(x,\tau)=u_\tau(x)&\textrm{in $\Omega$},
\end{array}\right.$$
where $\Omega\subset\mathbb R^N$ is a bounded domain, $\eta\in(0,1]$ is an indexing parameter of the family of perturbed problems, and the functions $a_\eta,$ $f_\eta$ and $h_\eta$ satisfies the standard assumptions on dissipative parabolic problems of reaction-diffusion type, that will be specified below.

Our goal in this paper is to analyze the behavior of attractors for the problems $(P_\eta)$ as $\eta\to0.$ Some preliminar robustness results for a family of problems as above can be found in \cite{ND16} (see also \cite{DCDSB18} for an improved regularity result). Nevertheless in those references the considered perturbations are strongly uniform and the limiting problem is autonomous. Both restrictions are actually quite unsatisfactory and unreal in practice and have been removed in this paper.

The structure of the paper is as follows. In Section \ref{s2}, we recall some known results about solutions for this family of nonlocal reaction-diffusion problems that allow us to define suitable dynamical systems (in particular the standard assumptions appear in (A1)). Since it is also well-known, we combine here some abstract results ensuring the existence of pullback attractors for multi-valued processes (which is the case here), to apply immediately to problems $(P_\eta)$ under the additional assumption (A2). In Section \ref{s3}, the robustness property is settled down step by step with some successive results introduced by suitable assumptions completing the previous ones. Theorem \ref{t19} is our main result and involves a fully non-autonomous development since the limiting problem (denoted $(P_0)$) is in general non-autonomous. The \emph{ad hoc} condition (A5) used in Theorem \ref{t19} is analyzed at the end of the paper. Some remarks and sufficient conditions to guarantee (A5) are provided, relating the forces $h_\eta$ and approppriate tempered parameters.
 
\section{Dynamical systems and attractors}\label{s2}

The notation $(\cdot,\cdot)$ will be used for the scalar product between elements in $L^2(\Omega )$ and also the duality between $L^p(\Omega )$ and $L^q(\Omega )$ ($1/p+1/q=1$). The open and closed balls in $L^2(\Omega )$ of center $a$ and radius $r$ are $B_{L^2(\Omega )}(a,r)$ and $\bar B_{L^2(\Omega )}(a,r)$ respectively. The dual of $H^1_0(\Omega)$ is  $H^{-1}(\Omega)$ and the duality product between them will be denoted by $\langle\cdot,\cdot\rangle.$ Here we identify $L^2(\Omega)$ with its dual, being the chain of embeddings $H^1_0(\Omega )\subset L^2(\Omega)\subset H^{-1}(\Omega )$ dense and compact. Let us also denote by $\lambda_1>0$ the first eigenvalue of $-\Delta$ with homogeneous Dirichlet boundary condition.

The existence of weak solutions for the problems $(P_\eta)$ requires the elements in the corresponding PDE to be in suitable spaces and with appropriate conditions. We collect all of them and present in the next assumption.
\begin{enumerate}
\item[(A1)] There exist positive constants $\alpha_1,$ $\alpha_2$ and $m$ and $\kappa_1\ge0,$ $\kappa_2\ge0,$ $p\ge2$ such that $\{a_\eta\}_{(0,1]}\subset C(\mathbb R;[m,\infty));$ $\{f_\eta\}_{(0,1]}\subset C(\mathbb R)$ fulfill
\begin{align*}
|f_\eta(s)|&\le\kappa_1+\alpha_1|s|^{p-1}\quad\forall s\in\mathbb R,\\
f_\eta(s)s&\le 
\kappa_2-\alpha_2|s|^p\quad\forall s\in\mathbb R,
\end{align*}
$\{h_\eta\}_{(0,1]}\subset L^2_{loc}(\mathbb R;H^{-1}(\Omega ))$ and $\{l_\eta\}_{(0,1]}\subset (L^2(\Omega ))'$ (since we make the identification of $L^2(\Omega)$ with its dual, for the rest of the paper we just put $\{l_\eta\}_{\eta\in(0,1]}\subset L^2(\Omega )$ and no distinction will be made between the classical notation in the precedent literature $l_\eta(v)$ or the form $(l_\eta,v)$ for each $v\in L^2(\Omega )$ and $\eta\in(0,1]$).
\end{enumerate}

\begin{Remark}
If $p=2,$ the dissipative condition can be weakened in the following way,
\begin{equation}\label{940} 
f_\eta(s)s\le\kappa_2+(m\lambda_1-\alpha_2)s^2.
\end{equation}
However, since this would lead to different expressions and conditions on the radii of the absorbing families, in order to simplify the exposition we prefer to keep (A1) as above. On the other hand, it is possible to use assumption (A1) with $1<p<2,$ but in such a case both conditions (the growth and the dissipative ones) are stronger than in the case where $p=2$ if we assume (\ref{940}).
\end{Remark}

\begin{Definition}
A weak solution to $(P_\eta)$ is an element $u\in L^\infty (\tau, T;L^2(\Omega ))\cap L^2(\tau,T;H^1_0(\Omega ))\cap L^p(\tau,T;L^p(\Omega ))$ for any $T>\tau$ with $u(\tau)=u_\tau$ and that verifies in the scalar distribution sense
$$\frac{d}{dt}(u,v)+a_\eta(l_\eta(u))(\nabla u,\nabla v)=(f_\eta(u),v)+\langle h_\eta,v\rangle\quad\forall v\in H^1_0(\Omega )\cap L^p(\Omega ).$$ 
\end{Definition}
\begin{Remark}\label{r1}
Since a weak solution $u$ satisfies that $u'\in L^2(\tau, T;H^{-1}(\Omega ))+L^q(\tau, T;L^q(\Omega ))$ for any $T>\tau,$ then $u\in C([\tau,\infty );L^2(\Omega ))$ and the following energy equality holds (e.g. cf. \cite[Lemma 3.2, p. 71]{Temam})
\begin{align*}
&|u(t)|^2+2\int_s^ta_\eta(l_\eta(u(r)))\|u(r)\|^2dr\\
=&|u(s)|^2+2\int_s^t((f_\eta(r),u(r))+\langle h_\eta(r),u(r)\rangle)dr\quad\forall \tau\le s\le t.
\end{align*}
\end{Remark}

The existence of global weak solutions for reaction-diffusion equations is well-known (e.g. cf. \cite{BV,Temam,CV}), and also when including nonlocal viscosity terms under the above assumptions (e.g. cf. \cite{ND16,PRSE18}; for a nonlocal $p$-Laplacian RD problem see \cite{JMAA18}), using compactness arguments.
\begin{Theorem}
Assume that (A1) holds. Then for any $u_\tau\in L^2(\Omega )$ there exists at least one weak solution to $(P_\eta).$ The set of weak solutions to $(P_\eta)$ with initial datum $u_\tau$ at time $\tau$ will be denoted by $\Phi_\eta (\tau,u_\tau).$
\end{Theorem}

Since no suitable Lipschitz (local or global) or monotonicity assumptions are imposed on the functions $a_\eta$ and $f_\eta$, we cannot guarantee a uniqueness result (for that e.g. cf. \cite[Th.2.1]{PRSE18}). This is interesting from the modeling point of view since continuous but non-differentiable viscosities may describe in a better way some diffusion processes, for instance in biological situations as bacteria accumulation or tumor growth (e.g. cf. \cite{cl1,chipot-lovat}). Nevertheless a multi-valued framework can be used to establish a suitable dynamical system.
\begin{Definition}
Given a metric space $(X,d),$ a multi-valued process $U$ on $X$ is a family of multi-valued maps $\mathcal U:\mathbb R^2_d\times X\to P(X)$, where $\mathbb R^2_d=\{(t,\tau)\in\mathbb R^2: t\ge\tau\}$ and $P(X)$ denotes the class of all nonempty subsets of $X,$ such that $\mathcal U(\tau,\tau,\cdot)=$Id$_X$ for any $\tau$ and $\mathcal U(t,r,x)\subset\mathcal U(t,s,\mathcal U(s,r,x))$ for any triplet $r\le s\le t$ and $x\in X.$ [When the inclusion relation becomes an equality, the process is said to be strict.]

A multi-valued process $U$ on $X$ is said upper semicontinuous if for all $(t,\tau)\in\mathbb R^2_d$ and any $x\in X$ and neighborhood $N$ of $U(t,\tau,x)$ in $X$ there exists a neighborhood $M$ of $x$ such that $U(t,\tau,y)\subset N$ for any $y\in M.$
\end{Definition}
The solution operator to $(P_\eta)$ allows us to define suitable multi-valued maps $U_\eta:\mathbb R^2_d\times L^2(\Omega)\to P(L^2(\Omega))$
$$(t,\tau,u_\tau)\mapsto U_\eta (t,\tau,u_\tau):=\{u(t):u\in\Phi_\eta(\tau,u_\tau)\}.$$
Moreover, the analogous compactness arguments on the solutions (and their continuity) give for the multi-valued maps the following result, whose proof is analogous to \cite[Lemma 1 and Proposition 1]{ND16}, so we omit it.
\begin{Proposition}\label{p5}
Assume that (A1) holds. Then $U_\eta$ is a strict multi-valued process on $L^2(\Omega )$ with closed values and upper semicontinuous for any $\eta\in(0,1]$.
\end{Proposition}

When a better description of stationary points and their stability or other dynamical properties is not available, it is still interesting to describe the existence of attractors, as general objects attracting the dynamics of solutions. The non-autonomous framework allows several interpretations, all of them useful, as uniform attractors, trajectory attractors, skew-product flows, etcetera. 

In this paper we focus on pullback (multi-valued) attractors, which describe the time-sections that attract solutions starting pullback in time. This \emph{pullback attracting property} roughly means that the studied phenomena started developing very long time ago. An interesting feature of this setting is that the class of attracted objects can increase those of fixed bounded sets, to larger classes called \emph{universes}, following the school of random dynamical systems. Another distinguished property is that more general assumptions can be supposed on the forces of the models in order to ensure the existence of pullback attractors and that minimal pullback attractors are contained in kernel sections of uniform attractors when the former exist. We briefly summarize the main concepts and ingredients, which of course are similar to those of the autonomous setting.

\begin{Definition}
A universe $\mathcal D$ in $X$ is a nonempty class of families parameterized in time $\widehat D=\{D(t)\}_{t\in\mathbb R}\subset P(X).$ A universe is said inclusion-closed if given $\widehat D=\{D(t)\}_{t\in\mathbb R}\in\mathcal D$ and $\widehat D'=\{D'(t)\}_{t\in\mathbb R}$ with $D'(t)\subset D(t)$ for any $t,$ then $\widehat D'\in\mathcal D.$

A family $\widehat D_0$ (not necessarily in $\mathcal D$) is pullback $\mathcal D$-absorbing for the process $U$ if for any $t\in\mathbb R$ and $\widehat D\in\mathcal D$ there exists $\tau(\widehat D,t)\le t$ such that $U(t,\tau,D(\tau))\subset D_0(t)$ for any $\tau\le\tau(\widehat D,t).$ 

A multi-valued process $U$ on $X$ is $\mathcal D$-asymptotically compact if for any $\widehat D\in\mathcal D,$ $t\in\mathbb R,$ and arbitrary sequences $\{\tau_n\}\subset (-\infty,t]$ with $\tau_n\to-\infty,$ $\{x_n\}\subset X$ with $x_n\in D(\tau_n)$ and $\{y_n\}$ with $y_n\in U(t,\tau_n,x_n),$ it holds that $\{y_n\}$ is relatively compact in $X.$
\end{Definition}

It is immediate that if a family $\widehat D_0\subset P(X)$ (not necessarily in $\mathcal D$) is pullback $\mathcal D$-absorbing for $U$ and also $U$ is $\widehat D_0$-asymptotically compact (with the analogous definition), then $U$ is $\mathcal D$-asymptotically compact. 

Actually these two main ingredients allow us to obtain the pullback attractor.
\begin{Definition}
A pullback $\mathcal D$-attractor for the multi-valued process $U$ on $X$ is any family $\mathcal{A_D}=\{\mathcal{A_D}(t)\}_{t\in\mathbb R}\subset P(X)$ such that (i) $\mathcal{A_D}(t)$ is a nonempty compact subset of $X$ for any $t\in\mathbb R;$ (ii) $\mathcal{A_D}$ pullback attracts any element of $\mathcal{D},$ that is, $\lim_{\tau\to-\infty}dist_X(U(t,\tau,D(\tau)),\mathcal{A_D}(t))=0$ for any $\widehat D\in\mathcal D$ and $t\in\mathbb R;$ (iii) $\mathcal{A_D}$ is negatively invariant, i.e. $\mathcal{A_D}(t)\subset U(t,\tau,\mathcal{A_D}(\tau))$ for any $\tau\le t.$
\end{Definition}
In general pullback attractors are not unique (cf. \cite{mrr09}). In order to gain uniqueness, we need to impose minimality: a pullback $\mathcal D$-attractor $\mathcal{A_D}$ is minimal if for any other family of closed sets $\widehat C=\{C(t)\}_{t\in\mathbb R}$ that pullback $\mathcal D$-attracts, it holds that $\mathcal{A_D}(t)\subset C(t)$ for any $t\in\mathbb R.$ As usual, the key objects in the sense of minimality are omega-limit sets (when they exist). Namely,
$$\Lambda(\widehat D,t)=\bigcap_{s\le t}\overline{\bigcup_{\tau\le s}U(t,\tau,D(\tau))}^X$$
is the omega-limit set of $\widehat D$ by $U$ at time $t.$ Actually, we can summarize all the above in the following result (cf. \cite[Theorem 2]{ND16}), extending the autonomous and non-autonomous multi-valued theories \cite{mv,clmv} to the framework of universes.
\begin{Theorem}\label{ea}
Consider an upper semicontinuous multi-valued process $U$ on a metric space $X$ with closed values, a universe $\mathcal D,$ a pullback $\mathcal D$-absorbing family $\widehat D_0=\{D_0(t)\}_{t\in\mathbb R}\subset P(X)$ and assume that $U$ is pullback $\mathcal D$-asymptotically compact. Then the family $\mathcal{A_D}=\{\mathcal{A_D}(t)\}_{t\in\mathbb R}$ given by
$$\mathcal{A_D}(t)=\overline{\bigcup_{\widehat D\in\mathcal D}\Lambda (\widehat D,t)}^X\quad\forall t\in\mathbb R$$
is the minimal pullback $\mathcal D$-attractor and $\mathcal{A_D}(t)\subset\overline{D_0(t)}^X$ for any $t\in\mathbb R.$ Moreover, if $\mathcal{A_D}\in\mathcal D$ and $U$ is a strict process, then $\mathcal{A_D}$ is (strictly) invariant under $U,$ i.e. $\mathcal{A_D}(t)=U(t,\tau,\mathcal{A_D}(\tau))$ for any $t\ge\tau.$
\end{Theorem}
\begin{Remark}\label{r9} (i) Actually for each $\widehat D\in\mathcal D,$ if $U$ is just pullback $\widehat D$-asymptotically compact, $\Lambda(\widehat D,t)$ is the minimal family of closed time-sections pullback attracting $\widehat D.$ Therefore these omega-limit families are the smallest pieces inside the attractor. 

\noindent (ii) If the pullback $\mathcal D$-absorbing family $\widehat D_0$ belongs to $\mathcal D,$ has closed sections and $\mathcal D$ is inclusion-closed, then $\mathcal{A_D}\in\mathcal D.$
\end{Remark}

After this brief recall, we get back to the family of problems $(P_\eta).$ The energy equality (cf. Remark \ref{r1}) and the Gronwall lemma give the following estimate.
\begin{Proposition}\label{p10}
Assume that (A1) holds. Then any weak solution $u$ to $(P_\eta)$ satisfies for any $\mu\in (0,2m\lambda_1)$
$$|u(t)|^2\le e^{-\mu (t-\tau)}|u_\tau|^2+\frac{2\kappa|\Omega |}{\mu}+\frac{1}{2(m-\mu(2\lambda_1)^{-1})}e^{-\mu t}\int_\tau^te^{\mu s}\|h_\eta(s)\|^2_*ds\quad\forall t\ge\tau.$$
\end{Proposition}
\begin{proof}
Any weak solution $u$ to $(P_\eta)$ satisfies the energy equality and by (A1)
$$\frac{1}{2}\frac{d}{dt}|u|^2+m|\nabla u|^2\le\kappa|\Omega|-\alpha_2\|u\|^p_p+\langle h_\eta,u\rangle\quad a.e.\ t>\tau.$$
Denoting $\delta>0$ such that $\mu=2(m-\delta)\lambda_1,$ after the H\"older and Young inequalities, 
\begin{equation}\label{3}
\frac{1}{2}\frac{d}{dt}|u|^2+(m-\delta)|\nabla u|^2+\alpha_2\|u\|^p_p\le\kappa|\Omega|+\frac{1}{4\delta}\|h_\eta\|^2_*\quad a.e.\ t>\tau.
\end{equation}
In particular, using the Poincar\'e inequality
$$
\frac{d}{dt}|u|^2+\mu|u|^2+2\alpha_2\|u\|^p_p\le 2\kappa|\Omega|+\frac{1}{2\delta}\|h_\eta\|^2_*\quad a.e.\ t>\tau,
$$
whence the Gronwall lemma concludes the proof.
\end{proof}

Now let us introduce one possible choice of tempered universe, suitable after the above estimate. This will kill the initial data leading to the absorbing property under an appropriate assumption on each force $h_\eta.$
\begin{Definition}
Given $\sigma>0,$ denote by $\mathcal D_\sigma^{L^2}$ the class of all families of nonempty subsets $\widehat D=\{D(t)\}_{t\in\mathbb R}$ such that $\lim_{\tau\to-\infty}e^{\sigma\tau}\sup_{v\in D(\tau)}|v|^2=0.$
\end{Definition}
The class of fixed bounded sets, i.e. families $\widehat D=\{D(t)\}_{t\in\mathbb R}$ with $D(t)=B$ bounded in $L^2(\Omega)$ for any $t\in\mathbb R$ is denoted by $\mathcal D_F^{L^2}.$

Observe that $\mathcal D_F^{L^2}\subset\mathcal D_\sigma^{L^2}$ for any $\sigma>0$ and that $\mathcal D_\sigma^{L^2}$ is inclusion-closed.

The following assumption allows to ensure the (pullback) absorbing property for suitable tempered universes.
\begin{enumerate}
\item[(A2)] There exist $\eta_0\in (0,1]$ such that for any $\eta\in (0,\eta_0]$ there exists $\mu_\eta\in(0,2\lambda_1m)$ with
$$\int_{-\infty}^0e^{\mu_\eta s}\|h_\eta(s)\|^2_*ds<\infty.$$
\end{enumerate}

The next result is an immediate consequence of Proposition \ref{p10} and (A2). Nevertheless the explicit expression (\ref{1325}) will be important later.
\begin{Corollary}\label{c12}
Assume that (A1)--(A2) hold. Then the process $U_\eta$ for $\eta\in(0,\eta_0]$ possesses a pullback $\mathcal D_{\mu_\eta}^{L^2}$-absorbing family $\widehat D_{0,\eta}=\{\bar B_{L^2(\Omega )}(0,R_\eta(t))\}_{t\in\mathbb R}$ with
\begin{equation}\label{1325}
R^2_\eta(t)=1+\frac{2\kappa|\Omega|}{\mu_\eta}+\frac{e^{-\mu_\eta t}}{2(m-\mu_\eta(2\lambda_1)^{-1})}\int_{-\infty}^t e^{\mu_\eta s}\|h_\eta(s)\|^2_*ds.
\end{equation}
Moreover, $\widehat D_{0,\eta}\in\mathcal D_{\mu_\eta}^{L^2}.$
\end{Corollary}

We end this summarizing section with the main result about existence of attractors.
\begin{Theorem}\label{t140}
Assume that (A1)--(A2) hold. Then each process $U_\eta$ for $\eta\in(0,\eta_0]$ possesses the minimal pullback $\mathcal D_{\mu_\eta}^{L^2}$-attractor $\mathcal A^\eta_{\mathcal D_{\mu_\eta}^{L^2}},$ which belongs to the universe $\mathcal D_{\mu_\eta}^{L^2}$ and is strictly $U_\eta$-invariant.

Moreover, the minimal pullback $\mathcal D_F^{L^2}$-attractors $\mathcal A^\eta_{\mathcal D_F^{L^2}}$ (for the corresponding processes $U_\eta$) also exist and the following relations hold
\begin{equation}\label{1420}
\mathcal A^\eta_{\mathcal D_F^{L^2}}(t)\subset\mathcal A^\eta_{\mathcal D_{\mu_\eta}^{L^2}}(t)\subset \bar B_{L^2(\Omega )}(0,R_\eta(t))\quad\forall t\in\mathbb R.
\end{equation}
\end{Theorem}
\begin{proof}
After Proposition \ref{p5} and Corollary \ref{c12}, the (pullback) asymptotic compactness is the remaining property in order to apply Theorem \ref{ea}. But this is an almost verbatim copy of \cite[Proposition 5]{ND16}. Observe that the attractor belongs to the corresponding universe since the absorbing family does and has closed sections (cf. Remark \ref{r9} (ii)).

The second part of the statement is a byproduct of the inclusion $\mathcal D_F^{L^2}\subset D_{\mu_\eta}^{L^2}$, the same abstract results applied to the universe of fixed bounded sets $\mathcal D_F^{L^2}$ and the minimality property of each attractor. 
\end{proof}

\section{Robustness result in $L^2$}\label{s3}

We are interested in the behavior of attractors for the dynamical systems associated to problems $(P_\eta)$ when $\eta\to 0.$ Our approach supposes that the problems $(P_\eta)$ represent approximations to a final problem and data are approaching after collecting more and more information toward the definitive values of the involved external and reaction forces and viscosity terms. Namely:
\begin{enumerate}
\item[(A3)] Assume that there exist elements $a_0,$ $f_0,$ $h_0$ and $l_0$ such that $\{a_\eta\},$ $\{f_\eta\},$ $\{l_\eta\}$ and $\{h_\eta\}$ fulfill as $\eta\to0$ that
\begin{align*}
a_\eta&\to a_0 \quad \textrm{uniformly on compact intervals},\\
l_\eta&\rightharpoonup l_0 \quad \textrm{weakly in $L^2(\Omega )$},\\
f_\eta&\to f_0 \quad \textrm{uniformly on compact intervals},\\
h_\eta&\rightharpoonup h_0 \quad \textrm{weakly in $L^2(\tau, T;H^{-1}(\Omega ))$ for any $\tau<T$}.
\end{align*}
\end{enumerate}
Roughly speaking, problems $(P_\eta)$ are perturbations of the limit problem 
$$(P_0)\left\{\begin{array}{ll}
\displaystyle\frac{\partial u}{\partial t}-a_0(l_0(u))\Delta u=f_0(u)+h_0(t)&\textrm{in $\Omega\times (\tau,\infty)$},\\
u=0&\textrm{on $\partial\Omega\times (\tau,\infty)$},\\
u(x,\tau)=u_\tau(x)&\textrm{in $\Omega$}.
\end{array}\right.$$
It is immediate that assumptions (A1) and (A3) imply that the elements in $(P_0),$ $a_0,$ $f_0,$ $h_0,$ satisfy the analogous condition (A1) with the same constants. This means that existence of solutions is guaranteed, $\Phi_0$ (analogous definition) is well-defined and a process $U_0$ can be associated. We do not know yet whether it possesses an attractor since (A2) is not inherited by the limit.

In two steps we will see that this limit is not only formal but rigorous for solutions in finite-time intervals and for attractors as well.

\begin{Theorem}\label{t14}
Assume that (A1) and (A3) hold. Let $\tau\in\mathbb R$ be given. Suppose that $u^{\eta_n}_\tau\rightharpoonup u_\tau$ weakly in $L^2(\Omega )$ as $\eta_n\to0$ and consider a sequence of weak solutions $u^{\eta_n}\in\Phi_{\eta_n}(\tau,u^{\eta_n}_\tau)$. Then there exist a subsequence (relabeled the same) and $u^0\in\Phi_0(\tau,u_\tau)$ such that for any $T>\tau$ the sequence $\{u^{\eta_n}\}$ converge to $u^0$ in several senses, namely weakly-star in $L^\infty(\tau,T;L^2(\Omega)),$ weakly in $L^2(\tau,T;H^1_0(\Omega ))\cap L^p(\tau,T;L^p(\Omega ))$, strongly in $L^2(\tau,T;L^2(\Omega )),$ with $\{f_{\eta_n}(u^{\eta_n})\}$ converging to $f_0(u^0)$ weakly in $L^q(\tau,T;L^q(\Omega )),$ $\{-a_{\eta_n}(l_{\eta_n}(u^{\eta_n}))\Delta u^{\eta_n}\}$ converging to $-a_0(l_0(u^0))\Delta u^0$ weakly in $L^2(\tau,T;H^{-1}(\Omega ))$ and $\{(u^{\eta_n})'\}$ converging to $(u^0)'$ weakly in $L^2(\tau,T;H^{-1}(\Omega ))+L^q(\tau,T;L^q(\Omega )).$
\end{Theorem}
\begin{proof}
We follow the same lines as in \cite[Theorem 4]{ND16}, by using uniform estimates, the Gronwall lemma, the Aubin-Lions Theorem and a diagonal argument to increase the final time $T.$ For short let us just give the main ideas.

The main estimates follow from (\ref{3}) in Proposition \ref{p10} after the Gronwall lemma. The uniform estimates and compactness arguments imply the claimed convergences in $L^\infty(\tau,T;L^2(\Omega))$ and $L^2(\tau,T;H^1_0(\Omega ))\cap L^p(\tau,T;L^p(\Omega ))$ of $u^{\eta_n}$ to a certain element $u^0.$ 

Observe that the boundedness of $\{u^{\eta_n}\}$ in $L^p(\tau,T;L^p(\Omega ))$ and (A1) gives that $\{f_{\eta_n}(u^{\eta_n}))\}$ is bounded in $L^q(\tau,T;L^q(\Omega )).$ From all above we deduce that 
$$\{(u^{\eta_n})'\}\textrm{ is bounded in $L^2(\tau,T;H^{-1}(\Omega ))+L^q(\tau,T;L^q(\Omega )).$}$$ 
Now the Aubin-Lions Theorem provides the convergence of (a subsequence of, but relabeled the same) $\{u^{\eta_n}\}$ to $u^0$ strongly in $L^2(\tau,T;L^2(\Omega ))$ and by the Dominated Convergence Theorem almost everywhere in $(\tau,T)\times\Omega.$

The former almost everywhere convergence jointly with the uniform convergence on compact subsets of $\mathbb R$ of $f_{\eta_n}$ towards $f_0$ means that $f_{\eta_n}(u^{\eta_n})$ also converge almost everywhere to $f_0(u^0).$ Indeed, fix $x\in\Omega$ such that $u^{\eta_n}(x)\to u^0(x).$ Consider a compact neighborhood of $u^0(x),$ $K=\bar B(u^0(x))\subset\mathcal N_{u^0(x)}.$ Then
$$|f_{\eta_n}(u^{\eta_n} (x))-f_0(u^0(x))|\le |f_{\eta_n}(u^{\eta_n} (x))-f_0(u^{\eta_n}(x))|+|f_0(u^{\eta_n}(x))-f_0(u^0(x))|\to 0\quad \textrm{as $\eta_n\to0$}.$$
This means (cf. \cite[Lemme 1.3, p.12]{Lions}) that $\{f_{\eta_n}(u^{\eta_n})\}\rightharpoonup f_0(u^0)$ weakly in $L^p(\tau, T;L^p(\Omega )).$

Analogously from the almost everywhere convergence $u^{\eta_n}(t)\to u^0(t)$ in $L^2(\Omega)$ and the weak convergence of $l_{\eta_n}\rightharpoonup l_0$ weak in $L^2(\Omega)$ we obtain 
$$|l_{\eta_n}(u^{\eta_n})-l_0(u^0)|=|(l_{\eta_n},u^{\eta_n})-(l_0,u^0)|\le |(l_{\eta_n},u^{\eta_n}-u^0)|+|(l_{\eta_n}-l_0,u^0)|\to 0\quad a.e.\ t,$$
which implies (as the argument with the sequence $\{f_{\eta_n}\}$) that $\{a_{\eta_n}(l_{\eta_n}(u^{\eta_n}))\}$ converge almost everywhere to $a_0(l_0(u^0)).$ Using the Dominated Convergence Theorem again it is not difficult to conclude that in fact
$$a_{\eta_n}(l_{\eta_n}(u^{\eta_n}))\to a_0(l_0(u^0))\quad\textrm{strongly in $L^2(\tau,T).$}$$
Since $\Delta u^{\eta_n}\rightharpoonup\Delta u^0$ weakly in $L^2(\tau,T;H^{-1}(\Omega )),$ from above we also deduce that
$$-a_{\eta_n}(l_{\eta_n}(u^{\eta_n}))\Delta u^{\eta_n}\rightharpoonup -a_0(l_0(u^0))\Delta u^0\quad\textrm{weakly in $L^2(\tau,T;H^{-1}(\Omega )).$}$$
Finally, we deduce that $u^0$ solves the limit problem $(P_0),$ since the equation is satisfied and the initial condition can be deduce in a standard way testing against $v\varphi$ with $v\in H^1_0(\Omega )\cap L^p(\Omega)$ and $\varphi\in H^1(\tau,T)$ with $\varphi(T)=0$ and $\varphi(\tau)\neq0.$ Integrating and comparing the resulting expressions of the problems $(P_\eta)$ and $(P_0),$ the weak convergence assumption $u^{\eta_n}_\tau\rightharpoonup u_0$ concludes that $u^0(\tau)=u_0$ and the proof is finished.
\end{proof}

Unfortunately the convergence almost everywhere in time of $u^{\eta_n}(t)\to u^0(t)$ strongly in $L^2(\Omega )$ does not seem enough for our purposes. It still remains to gain one convergence for our arguments. We impose a slightly stronger assumption than (A3) for the sequence $\{h_\eta\}$ converging to $h_0.$
\begin{enumerate}
\item[(A4)] One of next two options holds: either $h_\eta\to h_0$ strongly in $L^2(\tau,T;H^{-1}(\Omega )$ for any $\tau<T,$ or $h_\eta\rightharpoonup h_0$ weakly in $L^2(\tau,T;L^2(\Omega )$ for any $\tau<T.$
\end{enumerate}

\begin{Lemma}\label{l16}
Under the assumptions of Theorem \ref{t14} and (A4), the converging sequence $\{u^{\eta_n}\}$ obtained in Theorem \ref{t14} also satisfies $u^{\eta_n}(t)\to u^0(t)$ strongly in $L^2(\Omega )$ for all $t>\tau.$
\end{Lemma}
\begin{proof}
Fix $T>\tau$. We continue with the compactness arguments used in Theorem \ref{t14}. Now from the energy equality (cf. Remark \ref{r1}) we have that 
$$|u^{\eta_n}(s)|^2\le 2\kappa|\Omega|(s-r)+|u^{\eta_n}(r)|^2+2\int_r^s\langle h_{\eta_n}(\theta),u^{\eta_n}(\theta)\rangle d\theta\quad\forall\tau\le r\le s\le T$$
and similarly for the solution $u^0$ to $(P_0)$
$$|u^0(s)|^2\le 2\kappa|\Omega|(s-r)+|u^0(r)|^2+2\int_r^s\langle h_0(\theta),u^0(\theta)\rangle d\theta\quad\forall\tau\le r\le s\le T.$$
Consider the functions
$$J_{n}(s):=|u^{\eta_n}(s)|^2-2\kappa|\Omega|s-2\int_\tau^s\langle h_{\eta_n}(\theta),u^{\eta_n}(\theta)\rangle d\theta,$$
$$J_0(s):=|u^0(s)|^2-2\kappa|\Omega|s-2\int_\tau^s\langle h_0(\theta),u^0(\theta)\rangle d\theta.$$
They are non-increasing, continuous and thanks to the convergences proved in Theorem \ref{t14} and (A4) it holds that $J_n(s)\to J_0(s)$ almost everywhere on $(\tau,T).$ Under these conditions we can assure in fact that the convergence $J_n(s)\to J_0(s)$ holds for all $s\in(\tau,T]$ (for a detailed explanation of this argument see for instance the proof of \cite[Proposition 1]{ND16}). Moreover, since (A4) provides the convergence of the integral terms, we deduce that
\begin{equation}\label{1108}
\lim_{\eta\to0}|u^{\eta_n}(s)|^2=|u^0(s)|^2\quad\forall s\in(\tau,T].
\end{equation}
Since we already had that $\{u^{\eta_n}\}$ is bounded in $C([\tau,T];L^2(\Omega)),$ then $\{u^{\eta_n}(s)\}$ converges weakly in $L^2(\Omega)$ to some element. Actually we may identify this weak limit since $\{(u^{\eta_n})'\}$ is also bounded in $L^q(\tau, T;H^{-1}(\Omega )+L^q(\Omega))$ and the compact embedding $L^2(\Omega )\subset\subset H^{-1}(\Omega )$ implies that the Ascoli-Arzel\`a Theorem can be used. Namely $\{u^{\eta_n}\}$ converges to $u^0$ strongly in $C([\tau,T];H^{-1}(\Omega )+L^q(\Omega)).$ Thus we identify the weak limit
$$u^{\eta_n}(s)\rightharpoonup u^0(s)\quad\textrm{weakly in $L^2(\Omega)\quad\forall s\in[\tau,T]$}.$$
This weak limit and the convergence of the norms (\ref{1108}) give the result in $(\tau,T].$ But $T$ is arbitrary, so the same argument in $T+1,$ $T+2$ and successively and a diagonal procedure finishes the proof.
\end{proof}

Before establishing our main result we need to embed somehow the pullback absorbing families in a suitable family which should be also controlled by the dynamical system $U_0.$ This is an \emph{ad hoc} assumption to provide the most general conditions on the family of problems $(P_\eta)$ and relate them suitably to $(P_0).$

\begin{enumerate}
\item[(A5)] There exists $\mu_0\in (0,2m\lambda_1)$ such that
\begin{equation}\label{1251}
\int_{-\infty}^0e^{\mu_0s}\|h_0(s)\|_*^2ds<\infty
\end{equation}
and
\begin{equation}\label{1156}
\lim_{t\to-\infty}\limsup_{\eta\to0}\frac{e^{(\mu_0-\mu_\eta)t}}{m-\mu_\eta(2\lambda_1)^{-1}}\int_{-\infty}^te^{\mu_\eta s}\|h_\eta(s)\|^2_*ds=0.
\end{equation}
\end{enumerate}
\begin{Remark}\label{r17}
For any $c\in[0,\infty),$ the family of balls $\{B_{L^2(\Omega)}(0,\Psi_c(t))\},$ where
$$\Psi_c^2(t):=c+\limsup_{\eta\to0}\frac{e^{-\mu_\eta t}}{2(m-\mu_\eta(2\lambda_1)^{-1})}\int_{-\infty}^te^{\mu_\eta s}\|h_\eta(s)\|^2_*ds,$$
belongs to the universe $\mathcal D_{\mu_0}^{L^2}.$
\end{Remark}

Without assuming totally (A5), but just (\ref{1251}) we have the following
\begin{Corollary}
Assume that (A1), (A3) and (\ref{1251}) hold. Then $U_0$ possesses the minimal pullback $D^{L^2}_F$-attractor $\mathcal A^0_{\mathcal D^{L^2}_F}$ and the minimal pullback $D^{L^2}_{\mu_0}$-attractor $\mathcal A^0_{\mathcal D^{L^2}_{\mu_0}}$ and the following relation holds
$$\mathcal A^0_{\mathcal D_F^{L^2}}(t)\subset\mathcal A^0_{\mathcal D_{\mu_0}^{L^2}}(t)\subset\bar B_{L^2(\Omega )}(0,R_0(t))\quad\forall t\in\mathbb R,$$
where
$$R_0^2(t):=1+\frac{2\kappa|\Omega|}{\mu_0}+\frac{e^{-\mu_0 t}}{2(m-\mu_0(2\lambda_1)^{-1})}\int_{-\infty}^t e^{\mu_0s}\|h_0(s)\|^2_*ds.$$
\end{Corollary}
\begin{proof}
It is a consequence of Theorem \ref{ea}, analogously to the result for the perturbed problems (cf. Theorem \ref{t140}), and where the family $\{B_{L^2(\Omega)}(0,R_0(t))\}$ is pullback $\mathcal D^{L^2}_{\mu_0}$-absorbing for $U_0.$ Therefore, the last relation can be deduced in the same way as (\ref{1420}).
\end{proof}

Beyond the above attraction result, we may understand better the constructed \emph{ad hoc} condition (\ref{1156}) in (A5). Roughly speaking we do not impose a uniform bound for all the radii $R_\eta$ in terms of $R_0,$ but a bound in terms of the superior limit (i.e. when $\eta_n\to 0$ only a finite amount of them may escape, but they do not matter). The radii of absorbing families for $(P_\eta)$ problems are controlled by a tempered function $\Psi^2_c$ in $\mathcal D_{\mu_0}^{L^2}.$ In other words, these time-sections of the (perturbed) attractors are subsets of the time-section of a family which is attracted through $U_0$ to the limiting attractor $\mathcal A^0_{\mathcal D^{L^2}_{\mu_0}}.$

We can establish now our main robustness result.
\begin{Theorem}\label{t19}
Assume that (A1)--(A5) hold and that 
\begin{equation}\label{1742}
\liminf_{\mu\to0}\mu_\eta=:\underline\mu>0.
\end{equation}
Then the families $\{\mathcal A^\eta_{\mathcal D^{L^2}_{\mu_\eta}}\}$ converge upper semicontinuously to $\mathcal A^0_{\mathcal D^{L^2}_{\mu_0}}$ as $\eta\to0,$ i.e. 
$$\lim_{\eta\to0}\textrm{dist}_{L^2(\Omega )}(\mathcal A^\eta_{\mathcal D^{L^2}_{\mu_\eta}}(t),\mathcal A^0_{\mathcal D^{L^2}_{\mu_0}}(t))=0\quad\forall t\in\mathbb R.$$
\end{Theorem}
\begin{proof}
By contradiction, suppose that the thesis is false. Then there exist $\varepsilon>0,$ $t\in\mathbb R$ and a sequence $\eta_n\to0$ with 
\begin{equation}\label{1409}
\textrm{dist}_{L^2(\Omega )}(\mathcal A^{\eta_n}_{\mathcal D^{L^2}_{\mu_{\eta_n}}}(t),\mathcal A^0_{\mathcal D^{L^2}_{\mu_0}}(t))>\varepsilon.
\end{equation}

Now consider the family $\widehat{D}_{0,0}=\{D_{0,0}(t)\}_{t\in\mathbb R}$ where 
\begin{equation}\label{1735}
D_{0,0}(t):=B_{L^2(\Omega )}(0,\Psi_{c_0}(t))\quad\textrm{with $c_0=2+\frac{2\kappa|\Omega|}{\underline{\mu}}$}.
\end{equation}
The choice of the constant $c_0$ makes $\widehat{D}_{0,0}$ a kind of envelope of almost all the pullback $\mathcal D^{L^2}_{\mu_{\eta_n}}$-absorbing families by (\ref{1325}). Since this family belongs to $\mathcal D^{L^2}_{\mu_0}$ (cf. Remark \ref{r17}), there exists $\bar\tau:=\tau(t,\widehat{D}_{0,0},\varepsilon)<t$ such that
$$\textrm{dist}_{L^2(\Omega)}(U_0(t,\bar\tau,D_{0,0}(\bar\tau)),\mathcal A^0_{\mathcal D^{L^2}_{\mu_0}}(t))<\varepsilon /2.$$

By (\ref{1409}) we may select points $z^{\eta_n}\in\mathcal A^{\eta_n}_{\mathcal D^{L^2}_{\mu_{\eta_n}}}(t)$ such that  
\begin{equation}\label{2012}
\textrm{dist}_{L^2(\Omega )}(z^{\eta_n},\mathcal A^0_{\mathcal D^{L^2}_{\mu_0}}(t))>\varepsilon.
\end{equation}
By the negative invariance of the minimal pullback attractors
$$\mathcal A^{\eta_n}_{\mathcal D^{L^2}_{\mu_{\eta_n}}}(t)\subset U_{\eta_n}(t,\bar\tau,\mathcal A^{\eta_n}_{\mathcal D^{L^2}_{\mu_{\eta_n}}}(\bar\tau)).$$
Therefore each $z^{\eta_n}$ belongs to the trajectory of a weak solution, namely there exist $\{z^{\eta_n}_{\bar\tau}\}$ with $z^{\eta_n}_{\bar\tau}\in\mathcal A^{\eta_n}_{\mathcal D^{L^2}_{\mu_{\eta_n}}}(\bar\tau)$ and $u^{\eta_n}\in\Phi_{\eta_n}(\bar\tau,z^{\eta_n}_{\bar\tau})$ with $z^{\eta_n}=u^{\eta_n}(t)\in\mathcal A^{\eta_n}_{\mathcal D^{L^2}_{\mu_{\eta_n}}}(t).$ Moreover, since the pullback attractors are contained in the time-section of the absorbing families, we have
$$z^{\eta_n}_{\bar\tau}\in B_{L^2(\Omega )}(0,R_{\eta_n}(\bar\tau))\quad\forall{\eta_n}.$$
Since $\eta_n\to0,$ by the properties of superior and inferior limits, all but at most a finite amount of $\eta_n$ satisfy
$$R^2_{\eta_n}(\bar\tau)\le\Psi^2_{c_0}(\bar\tau).$$
Thus $\{z^{\eta_n}_{\bar\tau}=u^{\eta_n}(\bar\tau)\}$ is bounded and there exist a weakly converging subsequence $u^{\eta_n}(\bar\tau)\rightharpoonup u_{\bar\tau}\in D_{0,0}(\bar\tau)$ and $u^0\in\Phi_0(\tau,u_{\bar\tau})$ such that, by using Theorem \ref{t14} and Lemma \ref{l16}, we may take small enough $\bar\eta(\bar\tau,t,\varepsilon)$ with
$$|u^{\eta_n}(t)-u^0(t)|<\varepsilon/2\quad\forall\eta_n\le\bar\eta.$$

Finally the triangle inequality yields a contradiction with (\ref{2012}) since
$$\textrm{dist}_{L^2(\Omega)}(u^{\eta_n}(t),\mathcal A^0_{\mathcal D^{L^2}_{\mu_0}}(t))\le \textrm{dist}_{L^2(\Omega)}(u^{\eta_n}(t),u^0(t))+\textrm{dist}_{L^2(\Omega)}(u^0(t),\mathcal A^0_{\mathcal D^{L^2}_{\mu_0}}(t))<\varepsilon\quad\forall \eta_n\le\bar\eta.$$
\end{proof}

\begin{Remark}
Although the result is stated in the usual terms of attractors, the proof of the robustness result really shows that the family of attractors is upper semicontinuously converging toward $\Lambda^0 (\widehat{D}_{0,0},t),$ where $\widehat{D}_{0,0}$ is introduced in (\ref{1735}) and the upper script 0 in the omega-limit means w.r.t. to the dynamical system $U_0$ (cf. Remark \ref{r9} (i)).
\end{Remark}

It might be convenient to complete our main result, Theorem \ref{t19}, with some explanations, comments and conditions that imply that the required assumptions hold, in particular the \emph{ad hoc} condition (A5).

\begin{Remark}
Some robustness results in the literature conclude with an autonomous limiting problem $(P_0)$ when $\eta\to0,$ i.e. $h_0\equiv 0.$ Actually that is the case treated in \cite{ND16} since there $h_\eta=\eta h$ for a fixed element $h\in L^2_{loc}(\mathbb R;H^{-1}(\Omega ))$ fulfilling condition (23) in \cite{ND16}. Of course, these assumptions imply that (A4), (A5) and (\ref{1742}) hold immediately in the context of this paper and that one may take any $\mu_0\in (0,2m\lambda_1],$ where the right extreme value $2m\lambda_1$ in the interval is also valid since the problem $(P_0)$ is autonomous ($h_0$ does not exist; see the proof of Proposition \ref{p10}).
\end{Remark}

\begin{Remark}
(i) If $h_0$ fulfills (\ref{1251}) and $\mu_\eta=\mu_0$ for all $\eta<\!<1,$ then (\ref{1742}) is trivial and (\ref{1156}) reduces to
\begin{equation}\label{1836}
\lim_{t\to-\infty}\limsup_{\eta\to0}\int_{-\infty}^te^{\mu_0s}\|h_\eta(s)\|^2_*ds=0,
\end{equation}
which looks easier to check. For instance, if $h_\eta=\hat f(\eta)h$ for some $h$ (and $\mu_0\in(0,2m\lambda_1)$) fulfilling (\ref{1251}) and some function $\hat f$ with $\limsup_{\eta\to0}|\hat f(\eta)|<\infty,$ then (\ref{1836}) holds, and it yields (A5).

(ii) If some other argument allowed simplifying the expression (\ref{1156}) neglecting what appears in front of the integral and reducing it to compute
$$\lim_{t\to-\infty}\limsup_{\eta\to0}\int_{-\infty}^te^{\mu_\eta s}\|h_\eta(s)\|^2_*ds,$$
one should observe that, in general, the limit in time $t$ and the superior limit in $\eta$ do not conmute as one may check immediately with a real counterexample of functions with finite integral on $(-\infty,0)$ and compact support of fixed length moving to the left as $\eta\to0.$
\end{Remark}

We have discussed till now the case where $h_0\equiv 0$ and the case of $\mu_\eta$ fixed for all $\eta.$ Consider now the case when not only $h_\eta$ is converging to $h_0$ but also the associated $\{\mu_\eta\}$ (not necessarily constants in $\eta$) converges to some value: $\mu_\eta\to\mu_0.$ The next result gives sufficient conditions such that (\ref{1251}) holds.
\begin{Proposition}\label{p23}
Assume that (A2) holds, $h_\eta\rightharpoonup h_0$ weakly in $L^2(-M, 0;H^{-1}(\Omega ))$ for any $M>0,$ $\mu_\eta\to\mu_0$ and
\begin{equation}\label{1056}
\sup_{\eta\in(0,\eta_0]}\int_{-\infty}^0e^{\mu_\eta s}\|h_\eta(s)\|_*^2ds<\infty,
\end{equation}
where $\eta_0$ is given in assumption (A2). Then
$$\int_{-\infty}^0e^{\mu_0 s}\|h_0\|^2_*ds<\infty.$$
\end{Proposition}
\begin{proof}
Denote by $C>0$ to the supremum in (\ref{1056}). Since $\lim_\eta\mu_\eta=\mu_0,$ we have that
$$e^{\mu_\eta \cdot}h_\eta (\cdot)\rightharpoonup e^{\mu_0\cdot}h_0(\cdot)\quad\textrm{weakly in $L^2(-M,0;H^{-1}(\Omega ))$ for all $M>0.$}$$
By the properties of the weak limit one deduces that
$$\int_{-M}^0e^{\mu_0s}\|h_0(s)\|^2_*ds\le C\quad\forall M>0,$$
whence the result follows.
\end{proof}

Our aim now is to provide some sufficient conditions that also imply (\ref{1156}), which combined with the above gives (A5). Observe that the assumptions now are stronger that in the previous result.
\begin{Proposition}\label{p24}
Under the assumptions of Proposition \ref{p23}, if $\{\mu_\eta\}\subset (0,2m\lambda_1)$ and $\mu_0=\lim_\eta\mu_\eta\in(0,2m\lambda_1)$ and
$$h_\eta\to h_0\quad\textrm{strongly in $L^2(-M,0;H^{-1}(\Omega ))$ for any $M>0$},$$
and 
$$\limsup_{\eta\to0}\int_{-\infty}^0e^{\mu_\eta s}\|h_\eta(s)\|^2_*ds=\int_{-\infty}^0e^{\mu_0s}\|h_0(s)\|^2_*ds.$$
Then (A5) holds.
\end{Proposition}
\begin{proof}
Since (\ref{1251}) follows by Proposition \ref{p23}, we are concerned with (\ref{1156}). In fact, by the assumption that $\mu_0$ is neither 0 nor $2m\lambda_1,$ the convergence $\mu_\eta\to\mu_0$ and the properties of superior limits, it suffices to check that
\begin{equation}\label{1152}
\lim_{t\to-\infty}\limsup_{\eta\to0}\int_{-\infty}^te^{\mu_\eta s}\|h_\eta(s)\|^2_*ds=0.
\end{equation}
Consider an arbitrary sufficiently small value $\varepsilon>0$ and denote
$M^0_\varepsilon$ such that
$$\int_{-\infty}^{-M^0_\varepsilon}e^{\mu_0s}\|h_0(s)\|^2_*ds=\varepsilon.$$
From the hypotheses
$$\lim_{\eta\to0}\int_{-M}^0e^{\mu_\eta s}\|h_\eta(s)\|^2_*ds=\int_{-M}^0e^{\mu_0 s}\|h_0(s)\|^2_*ds\quad\forall M>0.$$
In particular, therefore
$$\limsup_{\eta\to0}\int_{-\infty}^{-M}e^{\mu_\eta s}\|h_\eta(s)\|^2_*ds\le\varepsilon\quad\forall M\ge M^0_\varepsilon.$$ 
This yields (\ref{1152}), therefore (A5) and finishes the proof.
\end{proof}

There might be sequences $\{\mu_{\eta_n}\}\subset(0,2m\lambda_1)$ related to $h_{\eta_n}$ such that (A2) holds but the sequence $\{\mu_{\eta_n}\}$ do not converge. In this case, we may consider the inferior limit (already introduced) and the superior limit of the tempered parameters
$$\underline{\mu}:=\liminf_{\eta\to0}\mu_\eta,\qquad\bar\mu:=\limsup_{\eta\to0}\mu_\eta.$$
\begin{Corollary}\label{c25}
Assume that $h_\eta\rightharpoonup h_0$ weakly in $L^2(-M,0;H^{-1}(\Omega ))$ for any $M>0$ and that $\{\mu_\eta\}\subset(0,2m\lambda_1)$ are such that (\ref{1056}) holds. Then
$$\int_{-\infty}^0e^{\underline{\mu}s}\|h_0(s)\|^2_*ds<\infty.$$
\end{Corollary}
\begin{proof}
It suffices to consider a sequence $\mu_{\eta_n}$ converging to $\underline{\mu}$ and apply the Proposition \ref{p23}.
\end{proof}
\begin{Remark}
Obviously the same result holds for $\bar\mu$ instead of $\underline{\mu},$ but observe that $\underline{\mu}s\ge\bar\mu s$ for $s\le 0,$ so the result above is stronger than the analogous with $\bar\mu.$
\end{Remark}
The reformulation of Proposition \ref{p24} in terms of inferior and superior limits require to control $\mu_\eta$ from above and from below, to take advantages of the the properties of the superior limit in (\ref{1156}). The new resulting condition is not as straightforward to check as before.
\begin{Proposition}
Suppose that (A2) is satisfied, $h_\eta\rightharpoonup h_0$ in $L^2(-M,0;H^{-1}(\Omega ))$ for any $M>0,$ 
$$0<\underline{\mu}=\liminf_{\eta\to0}\mu_\eta\le\limsup_{\eta\to0}\mu_\eta=\bar\mu<2m\lambda_1$$
and (\ref{1056}) and (\ref{1152}) hold. Then, for any $\epsilon>0$ such that $\mu_0:=\bar\mu+\epsilon<2m\lambda_1,$ (A5) is true.
\end{Proposition}
\begin{proof}
By Corollary \ref{c25} it is immediate that any value $\mu_0$ as in the statement makes that (\ref{1251}) holds true. In order to obtain (\ref{1156}), $\mu_0=\bar\mu+\epsilon,$ by the properties of the superior limit, is such that when $\eta\to0,$ all but at most a finite number of $\eta$ fulfill $\bar\mu+\epsilon>\mu_\eta.$ Therefore $\mu_0-\mu_\eta>0$ and so $e^{(\mu_0-\mu_\eta)t}\le 1.$ This, combined with the assumption $\bar\mu<2m\lambda_1$ and the choice of $\mu_0,$ also less than $2m\lambda_1,$ implies
$$\frac{e^{(\mu_0-\mu_\eta)t}}{m-\mu_\eta(2\lambda_1)^{-1}}\le\frac{1}{m-\mu_0(2\lambda_1)^{-1}}.$$
Therefore (\ref{1156}) is a consequence of (\ref{1152}). 
\end{proof}

\section*{Conclusions}

A robustness result has been stablished in $L^2(\Omega)$ for pullback attractors in suitable universes for perturbed problems $(P_\eta)$ towards the minimal pullback attractor of the corresponding limit problem $(P_0).$ This pullback attractor acts in another suitable universe with tempered parameter $\mu_0.$ The perturbed elements $\{a_\eta\},$ $\{f_\eta\},$ $\{l_\eta\}$ and $\{h_\eta\}$ (these last ones associated with coefficients $\{\mu_\eta\}$) satisfy the appropriate conditions (A1)--(A5) such that the result holds, being the ad hoc condition (A5) the most difficult to verify. Final comments to derive sufficient conditions such that (A5) hold are given. Namely, the tempered parameters $\{\mu_\eta\}$ are analyzed in several situations, fixed or not, converging or just with inferior and superior limits. The results extend previous ones where the limit problem was autonomous, that is, $h_0\equiv 0.$

\subsection*{Acknowledgments}

RC is a fellow of the FPU program of the Spanish Ministry of Education, Culture and Sport, reference FPU15/03080. This work has been partially supported by the projects PGC2018-096540-B-I00 (MICINN/FEDER, EU), PID2019-108654GB-I00 (MICINN), US-1254251 and P18-FR-4509 (Junta de Andaluc\'{\i}a/FEDER/US).

PMR thanks Professor Rodiak Figueroa-L\'opez for pointing out the interest of the robustness problem combined to non-autonomous limiting dynamics, held in the stimulating environment of the Summer Meeting on Differential Equations 2019 Chapter at ICMC-USP, S\~ao Carlos, SP, Brazil.

The authors would like to congratulate Professor Peter E. Kloeden on this occasion for his seventieth birthday. For PMR and JV it has been sincerely a great pleasure for so many years of joint collaboration, inspiration from him to us and friendship. We wish him many more years of fruitful relations both about mathematics and friendship.

\end{document}